\documentclass{amsart}
\usepackage{amsmath,amssymb,amsthm,latexsym}
\usepackage[all]{xy} 

\newcommand{\R}{\mathbb{R}}
\newcommand{\Rn}{\R^n}
\newcommand{\ttw}{\mathbb{T}^2}
\newcommand{\tn}{\mathbb{T}^n}
\newcommand{\kt}{K^2}
\newcommand{\aftw}{\mathrm{Aff}_{\Z}(\mathbb{R}^2)}
\newcommand{\gltw}{\mathrm{GL}(2,\Z)}
\newcommand{\coho}{\mathrm{H}^2(\kt;\Z^2_{\rho})}
\newcommand{\cohoo}{\mathrm{H}^2(\kt;\Z^2_{\rho_1})}
\newcommand{\cohot}{\mathrm{H}^2(\kt;\Z^2_{\rho_2})}
\newcommand{\cohon}{\mathrm{H}^2(\kt;\Z^2_{\rho_3})}
\newcommand{\cohom}{\mathrm{H}^2(\kt;\Z^2_{\rho_4})}
\newcommand{\cohoi}{\mathrm{H}^2(\kt;\Z^2_{\rho_i})}
\newcommand{\Z}{\mathbb{Z}}
\newcommand{\cotan}{\mathrm{T}^*}
\newcommand{\de}{\mathrm{d}}
\newcommand{\fib}{\xymatrix@1{ F \;\ar@{^{(}->}^-{\iota}[r] & M
    \ar[r]^-{\pi} & B}}
\newcommand{\fibpl}{\xymatrix@1{ \mathrm{H}_1(F,\Z) \;\ar@{^{(}->}[r] & \mathcal{P}
    \ar[r] & B}}
\newcommand{\fibrpl}{\xymatrix@1{ \mathrm{H}_1(F,\Z) \;\ar@{^{(}->}[r] & \mathcal{P}'
    \ar[r] & B}}
\newcommand{\affr}{\mathrm{Aff}_{\Z}(\Rn)}
\newcommand{\gln}{\mathrm{GL}(n,\Z)}
\newcommand{\zn}{\Z^n}

\theoremstyle{definition}
\newtheorem{defn}{Definition}[section]
\newtheorem{exm}{Example}[section]

\newtheorem{rk}{Remark}[section]

\theoremstyle{plain}
\newtheorem{thm}{Theorem}[section]
\newtheorem{lemma}{Lemma}[section]
\newtheorem{cor}{Corollary}[section]
\title{Classification of Lagrangian fibrations over a Klein Bottle}
\author{Daniele Sepe}
\address{School of Mathematics and Maxwell Institute for Mathematical
  Sciences, The University of Edinburgh, James 
  Clerk Maxwell Building, King's Buildings, Edinburgh, EH9 3JZ, UK}
\email{d.sepe@sms.ed.ac.uk}
\thanks{I would like to thank Ivan Kozlov for his insightful comments
  on an earlier draft of this paper}
\begin{document}
\begin{abstract}
  This paper completes the classification of regular Lagrangian
  fibrations over compact surfaces. \cite{misha} classifies regular Lagrangian fibrations over $ \ttw$. The main theorem in
  \cite{hirsch} is used to in order to classify
  integral affine structures on the Klein bottle $\kt$ and, hence, regular
  Lagrangian fibrations over this space.
\end{abstract}
\maketitle
\setcounter{tocdepth}{1}
\tableofcontents

\section{Introduction}\label{sec:introduction}
This paper classifies regular Lagrangian fibrations with compact
fibre whose base space is a Klein bottle $\kt$. 

\begin{defn}
A fibration $\fib$ is \emph{Lagrangian} if the total space $M$ is a symplectic manifold and
each fibre $F$ is a Lagrangian submanifold.

\end{defn}

\noindent
The
most elementary examples of Lagrangian fibrations arise from Hamiltonian dynamics, where the
natural setting is the cotangent bundle $\cotan B$ of a manifold $B$
with canonical symplectic form $\omega_0$. The
bundle 
$
\xymatrix@1{\Rn \; \ar@{^{(}->}[r] & (\cotan B, \omega_0) \; \ar[r] & B}$
is a Lagrangian fibration. \cite{misha} classifies regular Lagrangian fibrations with fibre $\R^n$ subject to
the constraint that the fibre be \emph{complete}.  More examples come
from completely integrable Hamiltonian systems. 

\begin{defn}
A \emph{completely integrable Hamiltonian system} on a symplectic
manifold $(M^{2n},\omega)$ is a set of $n$ functions $f_1, \ldots, f_n : M \to
\R$ satisfying the following conditions

\begin{itemize}
\item The functions are in involution - i.e. $\{f_i,f_j\}=0$ for all
  $i,j$, where $\{.,.\}$ is the Poisson bracket induced by the
  symplectic form $\omega$;
\item $\de f_1 \wedge \ldots \wedge \de f_n \neq 0$ almost everywhere.
\end{itemize}
\end{defn}

\noindent
The local properties of a completely integrable Hamiltonian system
near a regular fibre of the map $\mathbf{f}=(f_1,\ldots,f_n) : M \to
\Rn$ are well-known, as they are completely described by the
Liouville-Arnol'd-Mineur theorem (see \cite{arnold} for a proof).

\begin{thm}[Liouville-Arnold-Mineur] \label{thm:la}
  Let $\mathbf{f}=(f_1,\ldots,f_n): M \to \Rn$ be a completely integrable
  Hamiltonian system with $n$ degrees of freedom and let $x \in \Rn$
  be a regular value in the image of $\mathbf{f}$. Suppose that $\mathbf{f}^{-1}(x)$ has
  a compact, connected component, denoted by $F_x$. Then
  \begin{itemize}
  \item $F_x$ is a Lagrangian submanifold of $M$ and is diffeomorphic
    to $\tn$;
  \item there is a neighbourhood $V$ of $F_x$ in $M$ that is
    symplectomorphic to an open neighbourhood $W$ of $\tn$ in $\cotan
    \tn$, as shown in the diagram below.
    \begin{displaymath}
      \xymatrix{M \ar@{<-^{)}}[r] \ar@{<-^{)}}[dr] &\; V \ar[r]^-{\tilde{\varphi}} \ar[d]
        & W\; \ar[d] \ar@{^{(}->}[r] & \mathrm{T}^*\tn  \\
        & \; F_x \ar[r]^-{\varphi} & \tn \;\ar@{^{(}->}[ur] &}
    \end{displaymath}
  \end{itemize}
  
  \noindent
  Clearly, $W \cong D^n \times \tn$. Let $\tilde{\varphi}(p)=
  (a^1(p),\ldots,a^n(p),\alpha^1(p), \ldots, \alpha^n(p))$.
  The coordinates $a^i$ are called the \textbf{actions} and depend
  smoothly on the functions $f_i$. The coordinates $\alpha^i$ are called
  the \textbf{angles}.
\end{thm}

\noindent
In particular, in a neighbourhood of a regular value of
$\mathbf{f}$ a completely integrable Hamiltonian system is given by a
regular Lagrangian fibration. Conversely, given a regular Lagrangian
fibration it is given locally by a completely
integrable Hamiltonian system (\cite{dui}); Theorem
\ref{thm:la} therefore shows that the fibre of a regular Lagrangian
fibration with compact fibre is a torus $\tn$.\\

A natural question to ask is whether for a completely integrable Hamiltonian system the local action-angle coordinates given by Theorem \ref{thm:la} can be extended to global coordinates. In
\cite{dui} Duistermaat constructed obstructions to the existence of
such global coordinates for a regular Lagrangian fibration with
compact fibres. These obstructions are of a topological nature and are

\begin{itemize}
\item the \emph{monodromy} - obstruction for the fibration to be a
  principal $\tn$-fibration (Theorem \ref{thm:la} shows that any
  Lagrangian fibration is locally a principal $\tn$-fibration);
\item the \emph{Chern class} - obstruction to the existence of a
  global section.
\end{itemize}

\noindent
His motivation came from Cushman's work on the spherical pendulum, where monodromy naturally arises \cite{cush_sph}. In
\cite{zun_symp2} Zung introduced a \emph{symplectic} invariant,
called the \emph{Lagrangian class}, which is the obstruction to the
existence of a fibrewise symplectomorphism between a given Lagrangian
fibration and an appropriately defined reference Lagrangian fibration. His work is a natural continuation of the study of
isotropic fibrations carried out by Dazord and Delzant in
\cite{daz_delz}.

\begin{defn}
  An \emph{integral affine structure} on a manifold $B$ is a choice of
  atlas for $B$ in which all the transition functions lie in the group
  $\affr = \gln \ltimes \Rn$, where multiplication is defined as
  \begin{equation*}
    (A,x) \cdot (B,y) = (AB, x+Ay)
  \end{equation*}
\end{defn}

\noindent
The base space $B$ of a Lagrangian
fibration carries a natural integral affine structure \cite{bates_monchamp, cush_dav, dui, zun_symp2}.
This observation connects regular Lagrangian fibrations with the
theory of (integral) affine geometry, which is a very interesting and
rich field of mathematics and has been extensively studied \cite{hirsch, milnor, smillie}. There are
topological restrictions on the topology of
integral affine manifolds \cite{milnor}; this paper shows that the only compact
surfaces that admit integral affine structures are the 2-torus
$\ttw$ and the Klein bottle $\kt$. This is tantamount to saying that
these are the only compact surfaces that can be the base space of
regular Lagrangian fibrations (see Lemma \ref{lemma:easy} below). There are examples
of regular Lagrangian fibrations whose base space is a non-compact
$2$-manifold (e.g. the annulus or an open M\"obius strip).\\

\noindent
\cite{misha} classifies regular Lagrangian
fibrations whose base space is $\ttw$; this is obtained using the main result from
\cite{hirsch} and exploiting techniques developed in
\cite{sf}. This paper classifies regular Lagrangian
fibrations whose base space is a Klein bottle $\kt$, therefore completing the classification of regular Lagrangian fibrations over compact
surfaces. The classification is up to a fibrewise
symplectomorphism covering the identity on the base space; there is no
loss of generality in imposing this condition, since there exists a
fibrewise symplectomorphism 
\begin{equation*}
  \xymatrix{M \ar^-{F}[r] \ar[d] &M' \ar[d] \\
    B \ar^-{f}[r] & B}
\end{equation*}
\noindent
covering a diffeomorphism $f : B \to B$ on the base space if and only
if there is a fibrewise symplectomorphism
\begin{equation*}
  \xymatrix{ M \ar^-{\Phi}[0,2] \ar[1,1]& &f^* M' \ar[1,-1]\\
    & B}
\end{equation*}
\noindent
covering the identity on $B$. Furthermore this approach clarifies
the importance of the integral affine structure on the base space
$B$. Henceforth it is understood that by 'classification
of regular Lagrangian fibrations' we mean up to fibrewise
symplectomorphisms covering the identity on the base space unless
otherwise stated.
\\

The structure of the paper is as follows. Section \ref{sec:gener-lagr-fibr}
discusses generalities of Lagrangian fibrations and affine manifolds,
defining the invariants of the fibrations and highlighting the
interplay between the affine structure on the base manifold $B$ and
the monodromy of the fibration. The main result of this section is
that any integral affine Klein bottle is \emph{geodesically
  complete}, which is a consequence of the main theorem in \cite{hirsch}. This result is crucial to the rest of the present paper and is used in section \ref{sec:class-integr-affine} to
classify all possible integral affine structures on $\kt$ using
elementary methods. Section \ref{sec:expl-constr} constructs all possible
examples of regular Lagrangian fibrations $\xymatrix@1{F \;
  \ar@{^{(}->}[r] & M \ar[r] & \kt}$, using ideas from \cite{sf}.


\section{Generalities on Lagrangian Fibrations and Affine
  Manifolds}\label{sec:gener-lagr-fibr}
Throughout this paper a 'Lagrangian fibration' is
regular with compact fibres
whose base space $B$ is connected unless otherwise stated. This
section briefly goes through the
topological and symplectic invariants of Lagrangian fibrations and
then relates the monodromy to the integral affine structure on the base
space $B$ of a Lagrangian fibration. \cite{cush_dav,dui,misha,zun_symp2} provide more details about the construction of these invariants.

\subsection{Invariants of Lagrangian Fibrations}
Let us begin with the archetypal Lagrangian fibration with compact
fibre. Let $B$ be a manifold and let $(\cotan B, \omega_0)$ be its
cotangent bundle with canonical symplectic form $\omega_0$. Let $P \to
B $ be a discrete subbundle of maximal dimension of the cotangent
bundle such that $P \to B$ is locally spanned by closed forms. There are several names in the literature for this
discrete subbundle; \emph{period lattice} or \emph{period-lattice
  bundle} are the most common. The symplectic form
$\omega_0$ descends to a
well-defined symplectic form on the quotient $ \cotan B/ P \to B$ (which, by abuse of
notation, is denoted by $\omega_0$ too). 

\begin{defn} \label{defn:topref}
  The topological fibration $\cotan B / P \to B$ is the
  \emph{topological reference Lagrangian fibration with period-lattice
    bundle $P$}.
\end{defn}
\noindent
For a given manifold $B$ there may be many non-isomorphic
topological reference Lagrangian fibrations; the isomorphism type of
the bundle depends on the choice of period-lattice
bundle as the next example illustrates \cite{misha}. 

\begin{exm}
  \label{exm:mish}
  Let $\cotan \ttw$ be the cotangent bundle of the $2$-torus with
  canonical symplectic form. Then there exist period lattice bundles
  $P,P' \to \ttw$ such that the quotients $\cotan \ttw / P$, $\cotan
  \ttw /P'$ are diffeomorphic to $\mathbb{T}^4$ and to a Thurston-type
  manifold respectively. These two spaces are not homeomorphic since
  their first Betti numbers are distinct.
\end{exm}

Classifying integral affine structures on the Klein bottle allows to classify the distinct isomorphism types of admissible period-lattice bundles and, hence, to classify Lagrangian
fibrations over $\kt$. Fix a Lagrangian fibration $\fib$ and consider the natural map $\rho: \pi_1(B) \to \mathrm{Aut}(F) \cong \gln$ associated to the fibration.

\begin{defn} \label{defn:free_mon}
  $\rho$ is called the \emph{monodromy} of the
  Lagrangian fibration.
\end{defn}

\begin{rk}
The map $\rho$ is only defined up to conjugacy in $\gln$, since it depends upon the choice of basepoint $b \in B$. The conjugacy class of $\rho$ is called the \emph{free monodromy} of the fibration; it determines the isomorphism class of the period lattice bundle $P \to B$ associated to the fibration.
\end{rk}

Any topologically trivial Lagrangian fibration has
a global section (the image of the zero section $s_0 : B \to \cotan B$ under the
quotient is an example). Given a Lagrangian fibration $\fib$ with monodromy map $\rho
: \pi_1 (B) \to \gln$, a natural question to ask is
whether it admits a global section $ s: B \to M$. Consider the
following diagram

\begin{equation*}
  \xymatrix{ & M \ar[d]^{\pi} \\
B \ar@{.>}[ur]^{s} \ar[r]^{\text{id.}} & B}
\end{equation*}

\noindent
where $\text{id.}: B \to B$ denotes the identity map. Since the fibre $F$ is simple, obstruction theory shows that the obstruction to the existence of a global section
is an element $c \in \mathrm{H}^2 (B; \zn_{\rho})$.

\begin{defn} \label{defn:cc}
  $c$ is the \emph{Chern class} associated to the Lagrangian fibration.
\end{defn}

\noindent
The Chern class of Definition \ref{defn:cc} is just the generalisation of the Chern
class for principal $\tn$-bundles. The following lemma is proved in \cite{dui}.

\begin{lemma}
  Let $\fib$ be a Lagrangian fibration with monodromy $\rho : \pi_1(B)
  \to \gln$. This fibration is isomorphic to the topological reference
  Lagrangian fibration 
  $\cotan B / P \to B$ (where $P \to B$ is the period lattice bundle
  associated to the given fibration) if and only if the Chern class
  vanishes.
\end{lemma}

It turns out that the (free) monodromy and Chern class of a Lagrangian
fibration completely describe the \emph{topological} type of the
fibration, \emph{i.e.} these invariants are topologically sharp. It is natural to try to carry this point of
view over to the symplectic classification (the
classification of fibrewise symplectomorphisms) of such
fibrations. There are however some additional subtleties to take into
account, which have to do with the concrete choice of period lattice bundle
$P \to B$ as the following examples show.

\begin{exm}
  \label{exm:period}
  This example is due to Lukina \cite{luk}. Let $(x,y)$ be
  coordinates on $\cotan S^1$, where $y$ is the coordinate along the
  fibre. Define period lattice bundles $P$ and $P'$ by
  \begin{equation*}
    P = \{(x,y) \in \cotan S^1 | \, y \in \Z \} \qquad  P' = \{(x,y) \in
    \cotan S^1 |\, y \in 2\Z \}
  \end{equation*}
  Consider Lagrangian fibrations $\cotan S^1/P \to B$ and
  $\cotan S^1/P' \to B$ where the symplectic forms are obtained from the
  canonical symplectic form $\omega_0 = \de x \wedge \de y$ on $\cotan
  S^1$. The symplectic volumes of $\cotan S^1/P$ and $\cotan S^1/P'$
  are $1$ and $2$ respectively and so these spaces cannot be symplectomorphic.
\end{exm}

\begin{exm}
  This example elaborates on the previous one. The idea of Example \ref{exm:period} is to change the volume of the fibre, which is tantamount to
  asking that the two period lattice bundles $P \to S^1$ and $P' \to
  S^1$ do not have the same integral affine invariants. The volume is
  clearly one such invariant and it is a symplectic
  invariant. However, the following shows that it is not a
  \emph{complete} invariant. Let $\cotan \ttw$ be the cotangent bundle
  of $\ttw$ with coordinates $(\mathbf{x}, \mathbf{y})$ (where
  $\mathbf{x}$ are coordinates on $\ttw$) and canonical symplectic
  form. Define two period-lattice bundles $Q$ and $Q'$ by 
  \begin{equation*}
    Q = \{(\mathbf{x},\mathbf{y}) \in \cotan \ttw | \, \mathbf{y} \in \Z
    \times \Z \} \qquad  Q' = \{(\mathbf{x},\mathbf{y}) \in
    \cotan \ttw | \mathbf{y} \in 2\Z \times \frac{1}{2}\Z\}
  \end{equation*}
  The resulting Lagrangian fibrations $\cotan \ttw/ Q \to \ttw$ and
  $\cotan \ttw/ Q' \to \ttw$ cannot be fibrewise symplectomorphic
  (this can be seen in local coordinates), but the fibres have the
  same volume by construction.
\end{exm}
\noindent
The above examples show that isomorphic period-lattice bundles $P,
P' \to B$ that are not integral-affinely isomorphic  (\emph{i.e.} the
isomorphism on the fibre cannot be expressed by an integral matrix
which is invertible over the integers) yield symplectically
\emph{distinct} symplectic reference Lagrangian fibrations. 
Thus the definition of a symplectic reference Lagrangian fibration
with given period-lattice bundle is subtle as the explicit choice of period-lattice bundle (up to its integral
affine invariants) matters. Let a fibrewise symplectomorphism
covering the identity be given

\begin{equation*}
  \xymatrix{ M \ar^-{\Phi}[0,2] \ar[1,1]& & M' \ar[1,-1]\\
    & B}
\end{equation*}
\noindent
and fix an explicit choice of period-lattice bundle
$P \to B$ for the Lagrangian fibration $M \to B$. Then a local
calculation in action-angle coordinates shows that a necessary
condition for a fibrewise symplectomorphism covering the identity to exist is that the
period-lattice bundle $P' \to B$ associated to $M' \to B$ is precisely
the period lattice bundle $P \to B$, whereby 'precisely' indicates
that the explicit choices of period lattice bundles $P,P' \to B$ be
the same. In the rest of this paper, we fix an explicit choice of the
period-lattice bundle $P \to B$ thereby bypassing the above
difficulties.

\begin{defn} \label{defn:sympref}
  The fibration $(\cotan B/P, \omega_0) \to B$ is the
  \emph{symplectic reference Lagrangian fibration with period-lattice
    bundle $P$}
\end{defn}

Having fixed a choice of symplectic reference Lagrangian fibration, it is possible to tackle the symplectic classification. A symplectic reference Lagrangian fibration always
admits a global Lagrangian section, given by
the image of the zero section $s_0 : B \to \cotan B$ under the
quotient. The Lagrangian class measures the obstruction to the
existence of such a section. The next example shows how in general the notion of
isomorphism in the symplectic sense is distinct from the same notion
in the topological category.

\begin{exm} \label{exm:torus}
  Let $B = \ttw = \R^2/\Z^2$ be the torus with standard integral affine structure. Let $(x,y)$ be local coordinates $(x,y)$ and consider the
  Lagrangian fibration $ \pi: M = \ttw
  \times \ttw = \cotan \ttw /P \to B$ where
  $P \to B$ is given by the span of the (globally defined) $1$-forms
  $\de x$, $\de y$. Let $\omega_0$ be the symplectic form on $M$
  arising from the standard quotient construction outlined above and
  let $\omega_{\alpha} = \omega_0 + \pi^* \phi_{\alpha}$, where $\phi_{\alpha} = \alpha \de x
  \wedge \de y \in \mathrm{H}^2(\ttw; \R)$. It follows from the
  classification in \cite{misha} that the bundles $(M,\omega_0) \to B$
  and $(M,\omega) \to B$ are fibrewise symplectomorphic (over the
  identity map on $B$) if and only if $\alpha \in \Z$. In particular, $\alpha$ mod $1$ determines the class of the fibration up to fibrewise symplectomorphism. 
\end{exm}
\noindent
A description of the Lagrangian class in full generality can be found
in \cite{daz_delz, zun_symp2}. For the purpose of this paper, it suffices to notice that, for fixed monodromy representation $\rho$ and Chern class $c \in \mathrm{H}^2(B;\zn_{\rho})$, the space of symplectically distinct Lagrangian fibrations is a quotient of $\mathrm{H}^2(B;\R)$. In the case of the Klein bottle $\kt$ we have
$\mathrm{H}^2(\kt; \R) = 0$ and so there are no symplectic invariants
to be considered in the classification of Lagrangian fibrations whose
base space is $\kt$.

\subsection{Integral Affine Manifolds}
This subsection relates the integral affine structure inherited by the base
space $B$ of a Lagrangian fibration to the monodromy of the Lagrangian
fibration itself. In particular, the linear part of the integral
affine structure completely determines the monodromy and so, in order
to classify the possible monodromies of  Lagrangian fibrations with
base space $B$, it is enough to classify the possible integral affine
structures on $B$. The reader is referred to \cite{hirsch} and the
references therein for more details about affine geometry.\\

The following basic lemma is the starting point.

\begin{lemma} \label{lemma:easy}
  A manifold $B$ is the base space of a Lagrangian fibration if and
  only if it is an integral affine manifold.
\end{lemma}

\noindent
Showing that the
base space of a Lagrangian fibration inherits an integral affine
structure is an instructive exercise in dealing with Hamiltonian
$\tn$-actions, while the other direction is just a computation
in local coordinates to make sure that the symplectic form defined
locally is actually a global form. For more details, see
\cite{cush_dav,dui, zun_symp2}. \\

Given an integral affine manifold $B$, consider the associated \emph{affine monodromy} representation $\lambda :
\pi_1(B,b) \to \affr$ \cite{hirsch}, where $b \in B$ is a basepoint. Composing with the projection $\affr \to \gln$, obtain the \emph{linear monodromy} representation $\tilde{\rho}: \pi_1(B,b) \to \gln$. These maps depend on the choice of basepoint $b \in B$, but their conjugacy classes do not. $\tilde{\rho}$ is the linear monodromy of the flat, torsion free connection $\nabla$ that $B$ inherits from the integral affine structure. It is well-known that the linear monodromy of the integral affine structure induced on the base space of a Lagrangian fibration $B$ is equal to the monodromy of the fibration in the sense of Definition \ref{defn:free_mon} \cite{daz_delz,dui,zun_symp2}. This approach is very helpful, for instance, to see that trivial
monodromy implies that the Lagrangian fibration is a principal
$\tn$-bundle. Locally a Lagrangian fibration is a principal torus
bundle. If the monodromy is trivial, then so is the linear
holonomy of $\nabla$; parallel transport
the free and transitive local action using $\nabla$ to define a globally free and
transitive $\tn$-action.

\subsection{Affine Structures on $\kt$}
In order to classify Lagrangian fibrations whose base space is $\kt$
it is necessary to classify the integral affine structures that $\kt$
admits. While this is a hard task for dimensions greater than $2$,
it is a manageable problem for surfaces. \cite{misha} classifies
Lagrangian fibrations whose base space is $\ttw$; a key element in the proof is the main result in \cite{hirsch}, briefly outlined below. 
Let $B$ be a compact, affine $n$-dimensional manifold with nilpotent fundamental
group. Then a natural question to ask is what the possible affine
universal covers $\tilde{B}$ of $B$ are. Recall that there exists a local diffeomorphism $D: \tilde{B} \to \Rn$ called the \emph{developing map} \cite{fur}.

\begin{defn}
 An affine manifold $B$ is \emph{complete} if its developing map $D :
 \tilde{B} \to \Rn$ is a homeomorphism
\end{defn}
\noindent
Throughout this paragraph $B$ is compact and its fundamental group $\pi_1(B)$ is nilpotent. Compactness does not imply completeness as shown in \cite{fur}. However, if $B$ admits a
\emph{parallel volume form}, \emph{i.e.} a non-zero exterior volume
form on $\Rn$ which is invariant under the induced action of the
linear holonomy group, then $B$ is complete. These two properties are equivalent as shown in \cite{hirsch}. It is important to remark the importance of the
parallel volume form
geometrically. Its existence means that there are no expansions in the
linear holonomy group and this implies that the linear part of the affine action of $\pi_1(B)$
on $\Rn$ is actually \emph{unipotent}. This last implication relies
heavily on the assumption that $\pi_1(B)$ is nilpotent and that the
manifold $B$ is compact. The following simple example shows that the
result fails for non-compact manifolds with abelian fundamental group.

\begin{exm}
  Let $Y = \R^2 - \{0\}$, let $\pi = \pi_1(Y)$ denote its
  fundamental group and endow $Y$ with the integral affine structure
  arising from the inclusion $\xymatrix@1{Y \; \ar@{^{(}->}[r] &
    \R^2}$. Consider the following action
  \begin{align} \label{eq:17}
    \pi &\to \mathrm{GL}(2,\Z) \nonumber\\
    \zeta &\mapsto
    \begin{pmatrix}
      -1 & 0 \\
      0 & -1
    \end{pmatrix}
  \end{align}
  \noindent
  where $\zeta$ denotes a fixed generator of $\pi$. This action
  corresponds to considering the involution $-I : Y \to Y$. The above
  action is free and properly discontinuous, so that $Y/\langle -I \rangle$ is an
  integral affine manifold with linear (and affine) holonomy map given
  by (\ref{eq:17}). Note that $Y/\langle -I \rangle \cong Y$ and thus the above is an example of a non-unipotent integral affine structure on a
  non-compact manifold $Y$ with abelian fundamental group.
\end{exm}

A compact orientable integral affine
manifold has a parallel volume form, since the linear holonomy lies
entirely in $\mathrm{SL}(n,\Z)$. \cite{hirsch} proves the following

\begin{thm} \label{thm:torus}
  Any compact orientable integral affine manifold $B$ with nilpotent fundamental
  group is complete and the linear holonomy is unipotent.
\end{thm}

In particular the following corollary also holds and is proved in \cite{hirsch}

\begin{cor}
  Any integral affine structure on $\tn$ is complete.
\end{cor}

Theorem \ref{thm:kt} below shows how to extend Theorem \ref{thm:torus} to a more general
family of integral affine manifolds, namely those obtained as a
quotient of a compact integral affine manifold with nilpotent
fundamental group by some integral affine action.

\begin{thm} \label{thm:kt}
  Let $B$ be a compact integral affine manifold and suppose that there exists a regular covering $\hat{B} \to B$ of $B$ by a compact
  orientable integral affine manifold $\hat{B}$ whose fundamental
  group is nilpotent. Then any integral affine structure on $B$ is complete. 
\end{thm}
\begin{proof}
  Fix an integral affine structure on $B$ and pull it back to
  $\hat{B}$ via the regular covering $\hat{B} \to B$. Let $X$ be the
  integral affine universal cover of $B$ such that $D : X \to \Rn$
  induces the given integral affine structure. There is a commutative diagram 
  \begin{equation*}
    \xymatrix{ & X \ar[2,0] \ar[1,-1]\\
      \hat{B} \ar[1,1] & \\
      & B}
  \end{equation*}
  \noindent
  which commutes by naturality of the pullback of integral
  affine structures. In particular it shows that $X$ is an integral affine universal
  cover for both $\hat{B}$ and $B$ and $\hat{B}$ satisfies the hypothesis
  of Theorem \ref{thm:torus}. Hence $X$ is homeomorphic to $\Rn$
  with the standard integral affine structure, proving completeness of $B$.
\end{proof}

\begin{cor}\label{cor:klein}
  Any integral affine structure on $\kt$ is complete.
\end{cor}

\section{Classification of Integral Affine Structures on
  $\kt$}\label{sec:class-integr-affine}
Corollary \ref{cor:klein} simplifies the classification of integral affine structures
on $\kt$. Let $\Gamma$ denote the fundamental group of
$\kt$; it follows from Corollary \ref{cor:klein} that the classification of integral affine structures on $\kt$ amounts to classifying all injective homomorphisms
\begin{equation*}
  \Gamma \to \aftw
\end{equation*}
\noindent
inducing a free and properly discontinuous action of $\Gamma$ on
$\R^2$. Let $\Gamma$ have presentation $\Gamma
= \langle a,b \,| \, aba = b \rangle$ and let $G \lhd \Gamma$ be the
abelian subgroup generated by $a,b^2$. This subgroup corresponds to the fundamental
group of the two-fold cover $\ttw$ of $\kt$. By abuse of notation denote the images of the generators $a$, $b$ of $\Gamma$ in
$\aftw$ by $a$ and $b$ respectively. Write $a = (A,\mathbf{x})$ and $ b
= (B, \mathbf{y})$, where $A,B \in \gltw$ and $ \mathbf{x},\mathbf{y} \in \R^2$. The following
facts are true:
\begin{itemize}
\item since $a,b^2$ are generators for $\pi_1(\ttw)$, then $\det A = 1
  = \det B^2$;
\item $b$ is an orientation reversing transformation and hence $\det B
  = -1$;
\item the action of $\Gamma$ on $\R^2$ is free, so that $\mathbf{x},\mathbf{y} \neq 0$,
  as otherwise the origin is a fixed point of the action;
\item the action given by the composite $\xymatrix@1{\pi_1(\ttw) \;
    \ar@{^{(}->}[r] & \Gamma \ar[r] & \aftw}$ is also free and
  properly discontinuous since it defines the integral affine
  structure on $\ttw$ obtained by pulling back a fixed integral affine
  structure on $\kt$.
\end{itemize}

The following two basic lemmas are extremely useful and are proved in \cite{misha}.

\begin{lemma} \label{lemma:basic}
  If $(C,z) \in \aftw$ has no fixed point in $\R^2$, then the linear
  transformation $C$ has $1$ as one of its eigenvalues.
\end{lemma}

Lemma \ref{lemma:basic} is sometimes referred to in the literature as
'Hirach's principle' \cite{hirsch}. In some sense, the work done
in \cite{hirsch} generalises the above lemma to the case when the
affine manifold is compact and has nilpotent fundamental group.

\begin{lemma}\label{lemma:trace}
  For a matrix $D$ having $1$ as an eigenvalue, if $\det D = 1$
  we have that $\mathrm{Tr} D = 2$ and if $\det D = -1$ then
  $\mathrm{Tr} D =0$
\end{lemma}

By Lemma \ref{lemma:trace} the transformation $B$ satisfies $\det B =-1$ and $\mathrm{Tr}
B=0$. What matters in the following classification is simply the conjugacy class of $B$ in $\gltw$. There are only two such conjugacy classes and,
within each class, we are free to choose whichever representative
simplifies calculations the most. Take the following as initial
representatives for the two conjugacy classes, following \cite{sf}

\begin{equation*} \label{eq:100}
  B_1 =	\begin{pmatrix}
    1 & 0 \\
    0 & -1
  \end{pmatrix}
  \qquad
  \text{and}
  \qquad 
  B_2 = \begin{pmatrix}
    0 & 1 \\
    1 & 0
  \end{pmatrix}
\end{equation*}

\noindent
It remains to determine the matrices $A
\in \gltw$ with $\det A =1$, $\mathrm{Tr} A = 2$ satisfying the
relation $ABA = B$ where $B$ is one of $B_1, B_2$ in \ref{eq:100}. An
algebraic calculation shows that the only possibilities are
as follows

\begin{align*}
	A = I &, \qquad
	B_1 = \begin{pmatrix}
    1 & 0 \\
    0 & -1
  \end{pmatrix} \\
  A = A_1 =
  \begin{pmatrix}
    1 & n \\
    0 & 1
  \end{pmatrix}
  &,\qquad
  B_1 =
  \begin{pmatrix}
    1 & 0 \\
    0 & -1
  \end{pmatrix} \\
  A_2 =
  \begin{pmatrix}
    1 & 0 \\
    p & 1
  \end{pmatrix}
  &, \qquad
  B_1=
  \begin{pmatrix}
    1 & 0 \\
    0 & -1
  \end{pmatrix} \\
  A= I &, \qquad
  B_2 = \begin{pmatrix}
    0 & 1 \\
    1 & 0
  \end{pmatrix} \\
  A_3 =
  \begin{pmatrix}
    1 + n & n \\
    -n & 1 - n
  \end{pmatrix}
  &, \qquad
  B_2=
  \begin{pmatrix}
    0 & 1 \\
    1 & 0
  \end{pmatrix} \\
  A_4 = \begin{pmatrix}
    1 - n & n \\
    -n & 1 + n
  \end{pmatrix}
  &, \qquad
  B_2=
  \begin{pmatrix}
    0 & 1 \\
    1 & 0
  \end{pmatrix}  
\end{align*}
\noindent
where $p,n \neq 0$. \\

In order to determine the translation components of $a$ and $b$, impose the relation $aba=b$ once the linear part has been determined. It is possible to conjugate the generating set $a,b$  
by elements in
$\aftw$ and also to switch from generators $a,b$ to any other
generating set for $\Gamma$ satisfying the group relation. Let
$\mathbf{x}=(x_1,x_2)$ and $\mathbf{y}=(y_1,y_2)$ denote the
translation components of $a$ and $b$ respectively. 
The two-fold covering $\ttw \to \kt$ induces an injective
homomorphism $\pi_1(\ttw) \to \pi_1(\kt)$; the composite
map $\pi_1(\ttw) \to \pi_1(\kt) \to \aftw$ yields an integral affine
structure on $\ttw$, which is generated by $a$ and $b^2$ as above. When the linear part of $a$ is just the identity, then the
  translation components of $a, b^2$ need define a lattice for the
  action to be free. In particular, the
  matrix
  \begin{equation*}
    \begin{pmatrix}
      x_1 & 2y_1 \\
      x_2 & 0
    \end{pmatrix}
  \end{equation*}
  \noindent
  is non-degenerate and so $x_2, y_1 \neq 0$. Furthermore, the translation component of $b^2$ is non-trivial. It is necessary to check freeness of the action $\Gamma \to \aftw$. To
this end, it is useful to note that words in $\Gamma$ can be reduced to words of the form $a^q b^k$ for $k,q \in \Z$. In what follows, the explicit calculations appear in the first case only;
all other cases are similar in spirit but the calculations are more
cumbersome. Each case is dealt with separately and the corresponding generating set for the integral affine structure on the two-fold cover
$\ttw$ of $\kt$ is also indicated. Throughout the rest of the section, let $\mathbf{e}_1, \mathbf{e}_2$ denote the standard basis of $\R^2$. 


\subsection*{Case $A=I$, $B = B_1$}
The group relation on the translation components
  reduces in this case to $ B_1 \mathbf{x} = - \mathbf{x}$.
  Hence the translation component of $a$ has to lie in the
  $-1$-eigenspace of the linear part of $B$ and $x_1 = 0$. Conjugating
  both $a$ and $b$ by $(I,
      \frac{y_2}{2} \mathbf{e}_2)$ assume that $y_2 =0$. Changing generators to $a^{-1},b^{-1}$
  if necessary, assume further that $x_2, y_1 > 0$. Hence the generating set for these integral affine structures are of the form
  \begin{equation}
    \label{eq:5}
    a = (I, x\mathbf{e}_2), \qquad
    b= (B_1, y \mathbf{e}_1)
  \end{equation}
  \noindent
  where $x,y > 0$. It remains to check that this action is free. It is enough to consider the cases where the word
  in the group is of the form $a^q b^{2k+1}$ for $ k,q \in \Z$, since the case $a^q b^{2k}$ follows from \cite{misha}. Then 
  \begin{equation*}
    a^q = (I, qx \mathbf{e}_2), \qquad
    b^{2k+1}= (B_1, (2k+1)y\mathbf{e}_1)
  \end{equation*}
  \noindent
  for all $k,q \in \Z$. For the action to be free we want there to be no
  solution to $(a^q b^{2k+1}) \cdot \mathbf{z} = \mathbf{z} $ for $\mathbf{z} \in \R^2$. If this equation has a solution then $ 0 = -(2k+1)y$, which is absurd since $y \neq 0$. Hence the
  action is free. The induced structure on $\ttw$ is
   \begin{equation*}
    a = (I, x\mathbf{e}_2), \qquad
    b^2= (I, 2y \mathbf{e}_1)
  \end{equation*}
  \noindent
  where $x,y >0$.
\subsection*{Case $A=A_1$, $B=B_1$} First, since $a^{-1} b a^{-1}= b$, it is possible to change the
  generating set from $a,b$ to $a^{-1}, b$. Thus, switching $a$ with  $
  a^{-1}$ if necessary,
  assume that $n > 0$. Conjugating $a,b$ by $(I,
      -\frac{x_1}{n} \mathbf{e}_2)$
  it is possible to take $x_1=0$ and so $x_2 \neq 0$. Furthermore, conjugating both generators
  by $(-I, \mathbf{0})$ if necessary, choose $x_2 > 0$. Now the relation
  $aba=b$ on the translation components implies that 
  $ y_2 = x_2$ since $n \neq 0$. $a,b^{-1}$ form a generating set for $\Gamma$, so switching $b$ with 
  $b^{-1}$ if necessary, assume further that $y_1 >
  0$. Therefore generators for this family of homomorphisms are given by
  \begin{equation}
    \label{eq:4}
    a = (A_1, x\mathbf{e}_2), \qquad
    b= (B_1, y\mathbf{e}_1+x\mathbf{e}_2)
  \end{equation}
  \noindent
  where $x,y > 0$. Using the method outlined above, it can be checked
  that this action is free on $\R^2$. The
  essential aspect of these calculations is that both $x$ and $y$
  are different from $0$. The induced integral affine structure on
  $\ttw$ is given by
  \begin{equation*}
    a = (A_1, x\mathbf{e}_2)
    , \qquad
    b^2= (I, 2y\mathbf{e}_1)    
  \end{equation*}
  \noindent
  where $x,y > 0$ as above.
\subsection*{Case $A=A_2$, $B=B_1$} The group relation on translation components implies that    
 $x_1=0=y_1$. However this makes the translation
  component of $b^2$ equal to the zero vector and this
  is a contradiction.  Hence there are no admissible generating sets of
  this type, as the action is not free.\\

\begin{rk}  
  It is interesting to note that while the lower triangular case for $A$ cannot
  happen for $\kt$, it can for $\ttw$. Let $c,d$ denote generators for
  the fundamental group of
  $\ttw$. \cite{misha} shows that if the linear part of one of
  the generators is not diagonal, then the generating set
  is conjugate to one of the form
  \begin{equation*}
    c = \Bigg(
    \begin{pmatrix}
      1 & p \\
      0 & 1
    \end{pmatrix}
    ,
    \begin{pmatrix}
      0 \\
      w
    \end{pmatrix}
    \Bigg), \qquad
    d= \Bigg(
    \begin{pmatrix}
      1 & 0 \\
      0 & 1
    \end{pmatrix}
    ,
    \begin{pmatrix}
      z \\
      0
    \end{pmatrix}
    \Bigg)
  \end{equation*}
  \noindent
  where $p \in \Z$, $p,w,z > 0$. Conjugate $c,d$ by the
  reflection
  \begin{equation*}
    \Bigg(
    \begin{pmatrix}
      0 & 1 \\
      1 & 0
    \end{pmatrix}
    , \mathbf{0} \Bigg)
  \end{equation*}
  \noindent
  so that the linear part of $c$ is in lower triangular
  form. This is tantamount to reordering the basis with respect to which
  we consider the integral affine transformations $c,d$.
  The same argument in the $\kt$ case does not work because conjugation by the
  reflection above swaps the $1$ and $-1$ eigenspaces. This
  asymmetry is reflected geometrically in the fact that the above
  action is shown to be never free.
\end{rk}

\subsection*{Case $A=I$, $B=B_2$} Performing transformations as above, assume that the generating set is of the form
  \begin{equation}
    \label{eq:10}
    a = (I, x(\mathbf{e}_1-\mathbf{e}_2)),
    \qquad
    b= (B_2, y\mathbf{e}_2)
  \end{equation}
  \noindent
  where $x,y >0$. This action is also seen to be free and so it induces an integral affine structure on
  $\kt$. The induced action on $\ttw$ is given by
  \begin{equation*}
    a = (I, x (\mathbf{e}_1-\mathbf{e}_2)), \qquad
    b^2= (I, y(\mathbf{e}_1+\mathbf{e}_2))
  \end{equation*}
  \noindent
  where $x,y > 0$.

\subsection*{Case $A=A_3$, $B=B_2$} The relation $aba=b$ on the translation components implies that
  \begin{equation*}
    \begin{pmatrix}
      1+n & n \\
      1-n & -n
    \end{pmatrix}
    \begin{pmatrix}
      x_1 + x_2 \\
      y_1 + y_2
    \end{pmatrix}
    =
    \mathbf{0}
  \end{equation*}
  \noindent
  Since the determinant of the above matrix is $-2n$ and $n \neq 0$,
  it follows that $x_1 = -x_2$ and $y_1 = - y_2$. But if this is the
  case the vector component of $b^2$ is $\mathbf{0}$ and so the action is not free.

\subsection*{Case $A=A_4$, $B=B_2$} It is convenient to conjugate the
  generating set by the matrix
  \begin{equation*}
    \begin{pmatrix}
      1 & 0 \\
      -1 & 1
    \end{pmatrix}
  \end{equation*}
  \noindent
  so that the linear parts of $a$ and $b$ become
  \begin{equation}
    \label{eq:20}
    A_1 =
    \begin{pmatrix}
      1 & n \\
      0 & 1
    \end{pmatrix},
    \qquad
    B_3=
    \begin{pmatrix}
      1 & 1 \\
      0 & -1
    \end{pmatrix}
  \end{equation}
  \noindent
  As above, let $\mathbf{x}$ and $\mathbf{y}$ denote the vector
  components of the affine transformations $a$ and $b$. Conjugating
  the homomorphism by $(I, -\frac{x_1}{n}\mathbf{e}_2)$,
  assume that $x_1 = 0$. Switching to $a^{-1}$ if necessary, take $n > 0$. Furthermore, conjugating by $(-I,\mathbf{0})$
  if necessary, assume that $x_2 > 0$.  Applying the relation
  $aba=b$ to the vector components, the resulting generating sets are given by
  \begin{equation}
    \label{eq:21}
    a = (A_1, x\mathbf{e}_2)
    \qquad
    b= (B_3, y\mathbf{e}_1 + \frac{n-1}{n}x \mathbf{e}_2)
  \end{equation}
  \noindent
  where $x > 0$ and $y \in \R$ is arbitrary. However, a necessary condition for the action to be free is that
  $2y \neq - \frac{n-1}{n}x$. The induced integral affine structure on $\ttw$ is given by
  \begin{equation*}
    a = (A_1, x\mathbf{e}_2)
    , \qquad
    b^2= (I, (2y+\frac{n-1}{n}x)\mathbf{e}_1)
  \end{equation*}
  \noindent
  where $x>0$ and $2y \neq - \frac{n-1}{n}x$. \\
  \begin{rk}
    Some of the homomorphisms given by equation \eqref{eq:21} are
    conjugate to maps of equation \eqref{eq:4}, as pointed out to the
    author by Ivan Kozlov. In particular, a generating set of equation
    \eqref{eq:21} is conjugate to a homomorphism of equation
    \eqref{eq:4} if and only if the linear part of $a$ takes the
    form
    $$
    \begin{pmatrix}
      1 & 2l+1 \\
      0 & 1
    \end{pmatrix}
    $$
    \noindent
    for $l \in \Z$. Homomorphisms belonging to the family
    \begin{equation}
      \label{eq:2}
      a = \Bigg(
      \begin{pmatrix}
        1 & 2n \\
        0 & 1
      \end{pmatrix}
      , x\mathbf{e}_2 \Bigg)
      \qquad
      b= (B_3, y\mathbf{e}_1 + \frac{n-1}{n}x \mathbf{e}_2)
    \end{equation}
    \noindent
    cannot be conjugate to any of the homomorphisms of equations
    \eqref{eq:5} and \eqref{eq:10} because of the classification of
    integral affine structures on $\ttw$ carried out in \cite{misha}.
  \end{rk}
  Thus there are four families of integral affine structures on
  $\kt$. Families with distinct conjugacy classes of generating sets are
  not conjugate to one another and so integral affine structures on
  $\kt$ given by generating sets of different families are not
  isomorphic. The following theorem summarises the results of this
  section. 
  
  \begin{thm} \label{thm:class}
    The only possible homomorphisms $\Gamma \to \aftw$ inducing
    integral affine structures on $\kt$ are the ones shown in equations
    (\ref{eq:5}), (\ref{eq:4}), (\ref{eq:10}) and (\ref{eq:2}). Integral affine
    structures arising from distinct families are not be isomorphic to
    one another as integral affine structures.
  \end{thm}
  
  The above classification of integral affine Klein bottles gives
  insight into their geometry as well. Recall from section
  \ref{sec:gener-lagr-fibr} 
  that integral affine structures on a manifold $B$ are as
  flat, torsion-free connections with discrete linear holonomy (lying in $\gln$). If the linear
  holonomy consists of orthogonal matrices, then the connection is
  Levi-Civita. \cite{fur_2} shows that there is only one affine structure on the
  Klein bottle which induces a Levi-Civita connection up to affine
  diffeomorphism. The families given by equations (\ref{eq:5}) and
  (\ref{eq:10}) are not integral-affinely isomorphic and both induce a
  Levi-Civita connection on $\kt$. This should also be compared with the
  $\ttw$ case, where there is only one such family. The reason is that,
  up to conjugation in $\gltw$, there are two distinct integral affine 
  involutions on $\ttw$. 


\section{Explicit Constructions}\label{sec:expl-constr}
This section carries out the explicit constructions
of Lagrangian fibrations whose base space is the Klein bottle
$\kt$. The methods are similar to those in \cite{misha, sf}. The construction is naturally split in two parts: topological and symplectic.

\subsection{Topological Constructions}
As shown in Section \ref{sec:gener-lagr-fibr}, the only topological
invariants of Lagrangian fibrations are the monodromy and Chern
class. The classification of integral affine structures on $\kt$ given
by Theorem \ref{thm:class} gives all possible types of linear
monodromy that can arise. It is important to bear in mind that given
that the monodromy of a Lagrangian fibration is given by taking the
inverse transpose of the corresponding linear monodromy of the base
space (as an integral affine manifold. Set 
$$ A_2 =
\begin{pmatrix}
  1 & 2n \\
  0 & 1
\end{pmatrix}.
$$
\noindent
The monodromy of a Lagrangian fibration over the Klein bottle is then
given by one of the following four families
below 
\begin{equation*}
\rho_i(a) = 
\begin{cases}
I & \text{if $i=1$}; \\
(A_1^{-1})^T & \text{if $i=2$}; \\
I & \text{if $i=3$}; \\
(A_2^{-1})^T & \text{if $i=4$}.
\end{cases}
\qquad
\rho_i(b) =
\begin{cases}
(B_1^{-1})^T & \text{if $i=1$}; \\
(B_1^{-1})^T & \text{if $i=2$}; \\
(B_2^{-1})^T & \text{if $i=3$}; \\
(B_3^{-1})^T & \text{if $i=4$}.
\end{cases}
\end{equation*}
\noindent
where $A_1$ and $B_1,B_2,B_3$ are defined as in Section
\ref{sec:class-integr-affine} and $A_2$ is defined as above. For each such monodromy representation
$\rho_i$, there are Chern classes given by elements in the twisted
cohomology group $\cohoi$. It follows from \cite{daz_delz, zun_symp2}
that \emph{all} elements of $\cohoi$ arise as the Chern class of some
Lagrangian fibration with monodromy $\rho_i$. However, this result is
also proved by means of explicit constructions below.\\ 

The next lemma computes the twisted cohomology groups $\cohoi$. Let $n$ be the non-diagonal entry of $A_1$.

\begin{lemma}
  The twisted cohomology groups $\cohoi$ are given by
  \begin{equation*}
  \cohoi =
  \begin{cases}
  \Z/2
  \oplus \Z & \text{if } i=1; \\
  \Z/2 \oplus \Z/n & \text{if } i=2; \\
  \Z & \text{if } i=3 \\
  \Z/4n & \text{if } i=4.
  \end{cases}
\end{equation*}
\end{lemma}
\begin{proof}
  These calculations follow from standard methods in algebraic
  topology \cite{davis_kirk}. The standard cell decomposition of
  $\kt$
  \begin{equation*}
    \kt = e^0 \cup e^1_1 \cup e^1_2 \cup e^2
  \end{equation*}
  \noindent
  induces a $\Gamma$-equivariant cell decomposition on its universal
  cover $\tilde{\kt}=\R^2$, given by
  \begin{equation*}
    \R^2 = \bigcup_{\alpha \in \Gamma} (e^0_{\alpha} \cup
    e^1_{1,\alpha} \cup e^1_{2,\alpha} \cup e^2_{\alpha})
  \end{equation*}
  \noindent
  As a $\Z[\Gamma]$-basis for the module
  $\mathrm{C}_i(\tilde{\kt})$ take the following elements
  \begin{equation*}
    \begin{cases}
      e^0_{\alpha_0} & \text{if $i=0$} \\
      e^1_{1,\alpha_0},e^1_{2,\alpha_0} & \text{if $i=1$} \\
      e^2_{\alpha_0} & \text{if $i=2$} \\
    \end{cases}
  \end{equation*}
  \noindent
  for some fixed element $\alpha_0 \in \Gamma$. In order to simplify
  notation, forget about the
  $\alpha_0$ dependence.
  With respect to the above basis the following
  $\Z[\Gamma]$-equivariant chain complex
  \begin{equation*}
    \xymatrix@1{0 \ar[r] & \mathrm{C}_2(\tilde{\kt})
      \ar^-{\partial_2}[r] & \mathrm{C}_1(\tilde{\kt})
      \ar^-{\partial_1}[r] & \mathrm{C}_0(\tilde{\kt})
      \ar[r] & 0}
  \end{equation*}
  \noindent
  is defined by maps $\partial_2, \partial_1$ which are given on the basis
  elements by
  \begin{align*}\label{eqn:partial}
    \partial_2 e^2 =& (1+b) e^1_1 + (a-1)e^1_2 \\
    \partial_1 e^1_1 =& (a-1)e^0 \\
    \partial_1 e^1_2=& (ba - 1) e^0
  \end{align*}
  \noindent
  Using the $\Z[\Gamma]$-equivariant $\mathrm{Hom}$ functor, the following $\Z[\Gamma]$-equivariant cochain
  complex arises 
  \begin{equation*}
    \label{eq:11}
     \xymatrix@1{0 \ar[r] & \mathrm{Hom}(\mathrm{C}_0(\tilde{\kt}); \Z^2)
      \ar^-{\delta_1}[r] & \mathrm{Hom}(\mathrm{C}_1(\tilde{\kt});\Z^2)
      \ar^-{\delta_2}[r] & \mathrm{Hom}(\mathrm{C}_2(\tilde{\kt});\Z^2)
      \ar[r] & 0}
  \end{equation*}
  \noindent
  Since a $\Z[\Gamma]$-equivariant homomorphism
  $\mathrm{C}_2(\tilde{\kt}) \to \Z^2$ ($\Z^2$ is a
  $\Z[\Gamma]$-module via the representations $\rho_i$ for a fixed $i$) is
  completely determined by its value on the basis element
  $e^2$, the following holds
  \begin{equation*}
    \coho \cong \Z^2 /\mathrm{im} \,\delta_2
  \end{equation*}
  There are four cases to consider. Explicit calculations are given
  for the case $i=1$, all other ones are similar in spirit and
  therefore omitted. 
  \begin{description}
  \item[Case $\rho_1$] Let $\phi :\mathrm{C}_1(\tilde{\kt}) \to
    \Z^2$ be an equivariant homomorphism, then
    \begin{align*}
      (\delta_2 (\phi))(e^2) &= \phi(\partial_2 e^2) =
      (I+\rho_1(b))\phi(e^1_1) + (\rho_1(a) - I) \phi(e^1_2) \\
      &=
      \begin{pmatrix}
        2 & 0 \\
        0 & 0
      \end{pmatrix}
      \phi(e^1_1) =
      \begin{pmatrix}
        2s \\
        0
      \end{pmatrix}
    \end{align*}
    \noindent
    where $\phi(e^1_1) = (s,t) \in \Z^2$. Since the choice of $s$ is
    arbitrary, $\mathrm{im}\, \delta_2 \cong 2\Z \oplus 0$
    and we get that
    \begin{equation*}
      \label{eq:13}
      \cohoo \cong \Z/2 \oplus \Z
    \end{equation*}
  \item[Case $\rho_2$]
    In this case, $\mathrm{im}\, \delta_2 \cong 2\Z \oplus n\Z$, so that
    \begin{equation*}
      \cohot \cong \Z/2 \oplus \Z/n
    \end{equation*}
  \item[Case $\rho_3$]
    Following the methods above,  $\mathrm{im}\, \delta_2
    \cong \{(q,q):\; q \in \Z\} \cong \Z $, so that
    \begin{equation*}
      \cohon \cong \Z
    \end{equation*}
  \item[Case $\rho_4$]
    In this case, $\mathrm{im}\, \delta_2 \cong \Z \oplus 4n\Z$ and so
    $\cohom \cong \Z/4n$.
  \end{description}
\end{proof}

The following shows how to construct smoothly all Lagrangian fibrations over $\kt$. Take the trivial
fibration $(\R^2 \times \ttw,\omega_0) \to \R^2$ and, for a
given affine monodromy homomorphism $\tau_i : \Gamma \to \aftw$ whose linear part is denoted by $\rho_i$, define a
group action of $\Gamma$ on $\R^2 \times \ttw$ by
\begin{equation*}
  \tau \cdot (\mathbf{x},\mathbf{t}) = (\tau_i \cdot \mathbf{x}, \rho_i^*
  \cdot \mathbf{t}) 
\end{equation*}
\noindent
where $\rho_i^*$ denotes the inverse transpose of $\rho_i$. The above transformations are bundle automorphisms and they
preserve the symplectic form $\omega_0$. Hence the quotient bundle $
(\R^2 \times \ttw)/\Gamma \to \kt$ is a Lagrangian fibration with linear
monodromy given by $\rho_i$. Furthermore these fibrations have a
global section (the image of the zero section); this therefore
constructs all possible topological reference Lagrangian fibrations over $\kt$. Fix an affine monodromy representation $\tau_i : \Gamma \to \gltw$ and think of
$\kt$ as a square with opposite sides appropriately identified. Let $
\pi: (M, \omega) \to \kt$ be the Lagrangian fibration with linear monodromy given by $\rho_i$ and
zero Chern class. Cut out $\pi^{-1}(D_{2\epsilon})$ from $M$, where
$D_{2\epsilon}$ is a small closed disc of radius $2\epsilon$ in the
interior of the square. In
a neighbourhood of this disc the symplectic form is given by
\begin{equation*}
  \omega = \de x \wedge \de t_1 + \de y \wedge \de t_2
\end{equation*}
\noindent
where $x,y$ are local coordinates on a neighbourhood of $D_{2\epsilon}$
in $\kt$ and $t_1, t_2$ are coordinates on the fibre. Define a new
fibration $ M' \to \kt$ by
\begin{equation*}
  M = (M - (\mathrm{int}\, \pi^{-1}(D_\epsilon))) \cup_h \pi^{-1}(D_{2\epsilon})
\end{equation*}
\noindent
where the attaching map $h$ is defined on the intersection as
\begin{equation*}
  (x,y,t_1,t_2) \mapsto \Big(x,y,t_1 + \frac{m}{2\pi}\arg(x + \imath y)
  -\frac{n}{4\pi}f(x,y), t_2 + \frac{n}{2\pi}\arg(y+ \imath x) - \frac{m}{4\pi}f(y,x)\Big)
\end{equation*}
\noindent
where $f(x,y) = \frac{1}{2}\log\Big(\frac{x^2 + y^2}{1+y^2}\Big)$ and $\arg$ is
a well-defined map since it maps into an $S^1$ copy of the
fibre $F \cong \R^2/\Z^2$. Since $h$ is
a symplectomorphism (this can be checked explicitly by considering the
pullback of $\omega$ under $h$), the new bundle is still
Lagrangian and has Chern class depending on the pair $(m,n) \in
\Z^2$; some pairs $(m,n)$
will not change the topological structure of the original
bundle, as the resulting bundle still admits a global section. However, it is clear from the construction that \emph{all} Chern classes in $\cohoi$ can be realised and that the Chern class of
the bundle constructed above is the pair $(m,n)$ modulo appropriate
trivial pairs.\\

The above discussion proves the following

\begin{thm} \label{thm:topfinal}
The linear monodromy representations $\rho_i$ and the elements of the corresponding twisted cohomology groups $\cohoi$ classify regular Lagrangian fibrations with base space $\kt$ up to topological bundle isomorphism. For a fixed linear monodromy representation, \emph{all} elements of the corresponding cohomology group can be realised as the Chern class of some regular Lagrangian fibration.
\end{thm}

\subsection{Symplectic Constructions}
It remains to construct all possible regular Lagrangian fibrations over $\kt$ up to fibrewise symplectomorphism covering the identity on the base space. As shown in sections \ref{sec:introduction} and \ref{sec:gener-lagr-fibr}, there is no loss of generality in restricting to this this type of symplectomorphisms; furthermore, for this class of maps there are no issues with explicit choices of period-lattice bundles as shown in section \ref{sec:gener-lagr-fibr}. Fix an integral affine structure belonging to one of the families of Theorem \ref{thm:class}. As seen in the previous section, this fixes the monodromy representation of the regular Lagrangian classes over $\kt$ inducing this integral affine structure. Fix the topological type of such a Lagrangian fibration, \emph{i.e.} fix the Chern class. The space of Lagrangian classes for this fibration is a quotient of $\mathrm{H}^2(\kt;\R)$. Since $\mathrm{H}^2(\kt;\R)=0$ this space is trivial.  This discussion proves the following

\begin{thm}
For a fixed regular Lagrangian fibration over $\kt$ with given linear monodromy $\rho_i$ $(i=1,\ldots, 4)$ as in Theorem \ref{thm:class} and fixed Chern class, all symplectic structures preserving Lagrangianeity of this fibration are fibrewise symplectomorphic over the identity.
\end{thm}

This result is analogous to the main result in \cite{misha}. Fix an integral affine structure on $\ttw$ and the corresponding topologically reference Lagrangian fibration. The explicit choice of translation components for the generating set determines the explicit form of the subgroup $H$ of $\mathrm{H}^2(\ttw;\R)$ which acts trivially on the space of symplectic structures preserving Lagrangianeity of the fibration, as shown in \cite{misha}. Contrast this with the $\kt$ case. Lagrangian fibrations over $\kt$ inducing the same linear monodromy but distinct translation components are not going to be fibrewise symplectomorphic over the identity; however, the explicit choice of translation components does not matter when determining the space of Lagrangian classes for fixed linear monodromy as it is trivial in all cases.

\section{Conclusion}
This completes the classification of Lagrangian fibrations whose base
space is $\kt$. However, much more is yet to be known about
this type of fibrations in general. Even the case when the base is a
non-compact connected 2-surface is interesting to
consider. Furthermore, a classification of singular Lagrangian fibrations over surfaces would
be the natural next step to take. There exist examples of such
singular fibrations, like the famous $24$ focus-focus singularities
fibration over the sphere $S^2$ as mentioned in \cite{zun_symp2}. Better insight into this kind of
singular fibrations would be of interest and use to a wide range of
mathematicians, ranging from dynamicists to algebraic geometers and
people interested in mirror symmetry.

\bibliographystyle{abbrv}
\bibliography{mybib}
\end{document}